\documentclass{article}
\usepackage{amssymb}
\usepackage{amsmath}
\usepackage{amsfonts}
\usepackage{cite}

\newtheorem{theorem}{Theorem}

\newtheorem{remark}[theorem]{Remark}

\newenvironment{proof}[1][Proof]{\noindent\textbf{#1.} }{\ \rule{0.5em}{0.5em}}

\begin{document}

\title{A\ note on the order derivatives of Kelvin functions}
\author{J. L. Gonz\'{a}lez-Santander \and C/ Ovidi Montllor i Mengual 7,
pta. 9. \and 46017, Valencia, Spain.}
\date{}
\maketitle

\begin{abstract}
We calculate the derivative of the $\mathrm{ber}_{\nu }$, $\,\mathrm{bei}%
_{\nu }$, $\mathrm{ker}_{\nu }$, and $\,\mathrm{kei}_{\nu }$ functions with
respect to the order $\nu $ in closed-form for $\nu \in \mathbb{R}$. Unlike
the expressions found in the literature for order derivatives of the $%
\mathrm{ber}_{\nu }$ and $\,\mathrm{bei}_{\nu }$ functions, we provide much
more simple expressions that are also applicable for negative integral
order. The expressions for the order derivatives of the $\mathrm{ker}_{\nu }$
and $\,\mathrm{kei}_{\nu }$ functions seem to be novel. Also, as a
by-product, we calculate some new integrals involving the $\mathrm{ber}_{\nu
}$ and $\,\mathrm{bei}_{\nu }$ functions in closed-form. Finally, we include
a simple derivation of some integral representations of the $\mathrm{ber}$
and $\,\mathrm{bei}$ functions.
\end{abstract}



\section{Introduction}

In an investigation of the so-called \textit{skin effect} in a wire carrying
an alternating current, Lord Kelvin \cite{LordKelvin} introduced the $%
\mathrm{ber}$ and $\mathrm{bei}$ functions as the real and imaginary parts
respectively of the regular solution of the differential equation%
\begin{equation*}
xy^{\prime \prime }+y^{\prime }-ixy=0.
\end{equation*}

Lord Kelvin wrote this solution as%
\begin{equation*}
I_{0}\left( \sqrt{i}x\right) =\mathrm{ber}\left( x\right) +i\,\mathrm{bei}%
\left( x\right) .
\end{equation*}%
Since the main applications of the $\mathrm{ber}$ and $\mathrm{bei}$
functions comes from the analysis of the current distribution in circular
conductors \cite{Russel} where the argument $x$ represents the radial
coordinate, we will consider hereafter $x\geq 0$. For other applications see 
\cite[Chap. VIII-IX]{McLachlan}.

Whitehead \cite{Whitehead}\ generalized the functions introduced by Lord
Kelvin, denoting the real and imaginary parts of the solution of the
differential equation%
\begin{equation*}
x^{2}y^{\prime \prime }+xy^{\prime }-\left( \nu ^{2}+ix^{2}\right) y=0,
\end{equation*}%
as (see also \cite[Eqn. 10.61.1\&2]{DLMF})%
\begin{eqnarray}
\mathrm{ber}_{\nu }\left( x\right) +i\,\mathrm{bei}_{\nu }\left( x\right)
&=&e^{i\pi \nu }J_{\nu }\left( e^{-i\pi /4}x\right) ,  \label{ber_nu_def} \\
\mathrm{ker}_{\nu }\left( x\right) +i\,\mathrm{kei}_{\nu }\left( x\right)
&=&e^{-i\pi \nu /2}K_{\nu }\left( e^{i\pi /4}x\right) ,  \label{bei_nu_def}
\end{eqnarray}%
where $\nu \in 
%TCIMACRO{\U{211d} }%
%BeginExpansion
\mathbb{R}
%EndExpansion
$. In the literature \cite[Sect 10.61]{DLMF} we can find several properties
of the Kelvin functions $\mathrm{ber}_{\nu }$, $\,\mathrm{bei}_{\nu }$, $%
\mathrm{ker}_{\nu }$, and $\,\mathrm{kei}_{\nu }$. However, regarding the
integral representation of these functions, we have only found the results
given in \cite{Apelblat}, which read as%
\begin{eqnarray}
&&\mathrm{ber}_{\nu }\left( x\sqrt{2}\right)  \label{ber_nu_Int} \\
&=&\frac{1}{\pi }\int_{0}^{\pi }\left[ \cos \pi \nu \cos \left( x\sin t-\nu
t\right) \cosh \left( x\sin t\right) \right.  \notag \\
&&\qquad -\left. \sin \pi \nu \sin \left( x\sin t-\nu t\right) \sinh \left(
x\sin t\right) \right] dt  \notag \\
&&-\frac{\sin \pi \nu }{\pi }\int_{0}^{\infty }\exp \left( -\nu t-x\sinh
t\right) \cos \left( x\sin t+\pi \nu \right) dt,  \notag
\end{eqnarray}%
and%
\begin{eqnarray}
&&\mathrm{bei}_{\nu }\left( x\sqrt{2}\right)  \label{bei_nu_Int} \\
&=&\frac{1}{\pi }\int_{0}^{\pi }\left[ \cos \pi \nu \sin \left( x\sin t-\nu
t\right) \sinh \left( x\sin t\right) \right.  \notag \\
&&\qquad +\left. \sin \pi \nu \cos \left( x\sin t-\nu t\right) \cosh \left(
x\sin t\right) \right] dt  \notag \\
&&-\frac{\sin \pi \nu }{\pi }\int_{0}^{\infty }\exp \left( -\nu t-x\sinh
t\right) \sin \left( x\sin t+\pi \nu \right) dt,  \notag
\end{eqnarray}

In the Appendix, we obtain an alternative derivation quite simple for the
integral representation of the $\mathrm{ber}$ and $\,\mathrm{bei}$ functions
based on the idea given in \cite{SchulzDubois}. From the integral
representations (\ref{ber_nu_Int})\ and (\ref{bei_nu_Int}),\ the derivatives
with respect to the order are obtained in integral form as follows \cite%
{Apelblat}:\ 
\begin{eqnarray}
&&\frac{\partial \,\mathrm{ber}_{\nu }\left( x\right) }{\partial \nu }
\label{Dber_Int} \\
&=&\log \left( \frac{x}{2}\right) \mathrm{ber}_{\nu }\left( x\right) -\frac{%
3\pi }{4}\mathrm{bei}_{\nu }\left( x\right)  \notag \\
&&-\frac{x}{2\sqrt{2}}\int_{0}^{1}u^{\left( \nu -1\right) /2}\left[ \gamma
+\log \left( 1-u\right) \right] \left[ \mathrm{ber}_{\nu -1}\left( x\sqrt{u}%
\right) +\mathrm{bei}_{\nu }\left( x\sqrt{u}\right) \right] du,  \notag
\end{eqnarray}%
and%
\begin{eqnarray}
&&\frac{\partial \,\mathrm{bei}_{\nu }\left( x\right) }{\partial \nu }
\label{Dbei_Int} \\
&=&\log \left( \frac{x}{2}\right) \mathrm{bei}_{\nu }\left( x\right) +\frac{%
3\pi }{4}\,\mathrm{ber}_{\nu }\left( x\right)  \notag \\
&&+\frac{x}{2\sqrt{2}}\int_{0}^{1}u^{\left( \nu -1\right) /2}\left[ \gamma
+\log \left( 1-u\right) \right] \left[ \mathrm{ber}_{\nu -1}\left( x\sqrt{u}%
\right) -\mathrm{bei}_{\nu }\left( x\sqrt{u}\right) \right] du.  \notag
\end{eqnarray}

Recently in \cite{BrychovNew},\ we find closed-form expressions for the
order derivatives of the $\mathrm{ber}_{\nu }$ and $\,\mathrm{bei}_{\nu }$
functions, which read as 
\begin{eqnarray}
&&\frac{\partial \,\mathrm{ber}_{\nu }\left( x\right) }{\partial \nu }
\label{Dber_Brychov} \\
&=&\left[ \log \left( \frac{x}{2}\right) -\psi \left( \nu \right) -\frac{1}{%
2\nu }\right] \mathrm{ber}_{\nu }\left( x\right) -\frac{3\pi }{4}\,\mathrm{%
bei}_{\nu }\left( x\right)  \notag \\
&&+\frac{\pi \csc \pi \nu }{2\,\Gamma ^{2}\left( \nu +1\right) }\left( \frac{%
x}{2}\right) ^{2\nu }\left[ \sin \left( \frac{3\pi \nu }{2}\right) \mathrm{%
bei}_{-\nu }\left( x\right) -\cos \left( \frac{3\pi \nu }{2}\right) \mathrm{%
ber}_{-\nu }\left( x\right) \right] c\left( \nu ,x,0\right)  \notag \\
&&+\frac{\pi \nu \csc \pi \nu }{\Gamma ^{2}\left( \nu +2\right) }\left( 
\frac{x}{2}\right) ^{2\nu +2}\left[ \cos \left( \frac{3\pi \nu }{2}\right) 
\mathrm{bei}_{-\nu }\left( x\right) +\sin \left( \frac{3\pi \nu }{2}\right) 
\mathrm{ber}_{-\nu }\left( x\right) \right] c\left( \nu ,x,1\right)  \notag
\\
&&-\frac{x^{2}}{4\left( 1-\nu ^{2}\right) }\left[ \mathrm{bei}_{\nu }\left(
x\right) d\left( \nu ,x,0\right) +\frac{3x^{2}}{8\left( 4-\nu ^{2}\right) }\,%
\mathrm{ber}_{\nu }\left( x\right) d\left( \nu ,x,1\right) \right] ,  \notag
\end{eqnarray}%
and%
\begin{eqnarray}
&&\frac{\partial \,\mathrm{bei}_{\nu }\left( x\right) }{\partial \nu }
\label{Dbei_Brychov} \\
&=&\left[ \log \left( \frac{x}{2}\right) -\psi \left( \nu \right) -\frac{1}{%
2\nu }\right] \mathrm{bei}_{\nu }\left( x\right) +\frac{3\pi }{4}\,\mathrm{%
ber}_{\nu }\left( x\right)  \notag \\
&&-\frac{\pi \csc \pi \nu }{2\,\Gamma ^{2}\left( \nu +1\right) }\left( \frac{%
x}{2}\right) ^{2\nu }\left[ \cos \left( \frac{3\pi \nu }{2}\right) \mathrm{%
bei}_{-\nu }\left( x\right) +\sin \left( \frac{3\pi \nu }{2}\right) \mathrm{%
ber}_{-\nu }\left( x\right) \right] c\left( \nu ,x,0\right)  \notag \\
&&-\frac{\pi \nu \csc \pi \nu }{\Gamma ^{2}\left( \nu +2\right) }\left( 
\frac{x}{2}\right) ^{2\nu +2}\left[ \cos \left( \frac{3\pi \nu }{2}\right) 
\mathrm{ber}_{-\nu }\left( x\right) -\sin \left( \frac{3\pi \nu }{2}\right) 
\mathrm{bei}_{-\nu }\left( x\right) \right] c\left( \nu ,x,1\right)  \notag
\\
&&+\frac{x^{2}}{4\left( 1-\nu ^{2}\right) }\left[ \mathrm{ber}_{\nu }\left(
x\right) d\left( \nu ,x,0\right) -\frac{3x^{2}}{8\left( 4-\nu ^{2}\right) }\,%
\mathrm{bei}_{\nu }\left( x\right) d\left( \nu ,x,1\right) \right] ,  \notag
\end{eqnarray}%
and where the following functions have been defined:%
\begin{eqnarray*}
&&c\left( \nu ,x,a\right) \\
&=&\,_{3}F_{6}\left( \left. 
\begin{array}{c}
\frac{2\nu +a+1}{4},\frac{2\nu +3}{4},\frac{2\nu +5a}{4} \\ 
a+\frac{1}{2},\frac{\nu +a+1}{2},\frac{\nu +a}{2}+1,\frac{\nu +a}{2}+1,\nu +%
\frac{a+1}{2},\nu +1+\frac{a}{2}%
\end{array}%
\right\vert -\frac{x^{4}}{16}\right) ,
\end{eqnarray*}%
and%
\begin{eqnarray*}
&&d\left( \nu ,x,a\right) \\
&=&\,_{4}F_{7}\left( \left. 
\begin{array}{c}
\frac{a+1}{2},\frac{a+1}{2},\frac{2a+3}{4},\frac{2a+5}{4} \\ 
a+\frac{1}{2},\frac{a+3}{2},\frac{a+3}{2},\frac{\nu +a}{2}+1,\frac{\nu +a+3}{%
2},\frac{a-\nu }{2}+1,\frac{a-\nu +3}{2}%
\end{array}%
\right\vert -\frac{x^{4}}{16}\right) .
\end{eqnarray*}

It is worth noting that (\ref{Dber_Brychov})\ and (\ref{Dbei_Brychov})\
cannot be applicable for integral order, i.e. $\nu =n\in 
%TCIMACRO{\U{2124} }%
%BeginExpansion
\mathbb{Z}
%EndExpansion
$. Nevertheless, for $n\geq 0$, we find in the literature the following
expressions \cite[Eqn. 1.17.2(1)\&(2)]{BrychovBook}:%
\begin{eqnarray}
&&\left. \frac{\partial \,\mathrm{ber}_{\nu }\left( x\right) }{\partial \nu }%
\right\vert _{\nu =n}=-\frac{\pi }{2}\mathrm{bei}_{n}\left( x\right) -%
\mathrm{ker}_{n}\left( x\right)  \label{DBern_Brychov} \\
&&+\frac{n!}{2}\sum_{k=0}^{n-1}\frac{\left( z/2\right) ^{k-n}}{k!\left(
n-k\right) }\left[ \cos \left( \frac{5\left( k-n\right) \pi }{4}\right) 
\mathrm{ber}_{k}\left( x\right) +\sin \left( \frac{5\left( k-n\right) \pi }{4%
}\right) \mathrm{bei}_{k}\left( x\right) \right] ,  \notag
\end{eqnarray}%
and%
\begin{eqnarray}
&&\left. \frac{\partial \,\mathrm{bei}_{\nu }\left( x\right) }{\partial \nu }%
\right\vert _{\nu =n}=\frac{\pi }{2}\mathrm{ber}_{n}\left( x\right) -\mathrm{%
kei}_{n}\left( x\right)  \label{DBein_Brychov} \\
&&+\frac{n!}{2}\sum_{k=0}^{n-1}\frac{\left( z/2\right) ^{k-n}}{k!\left(
n-k\right) }\left[ \cos \left( \frac{5\left( k-n\right) \pi }{4}\right) 
\mathrm{bei}_{k}\left( x\right) -\sin \left( \frac{5\left( k-n\right) \pi }{4%
}\right) \mathrm{ber}_{k}\left( x\right) \right] .  \notag
\end{eqnarray}

Also, for the $\mathrm{ker}_{\nu }$ and $\,\mathrm{kei}_{\nu }$ functions,
we have for $n\geq 0$ \cite[Eqn. 1.17.2(5)\&(6)]{BrychovBook}: 
\begin{eqnarray}
&&\left. \frac{\partial \,\mathrm{ker}_{\nu }\left( x\right) }{\partial \nu }%
\right\vert _{\nu =n}=\frac{\pi }{2}\mathrm{kei}_{n}\left( x\right)
\label{DKern_Brychov} \\
&&+\frac{n!}{2}\sum_{k=0}^{n-1}\frac{\left( z/2\right) ^{k-n}}{k!\left(
n-k\right) }\left[ \cos \left( \frac{3\left( k-n\right) \pi }{4}\right) 
\mathrm{ker}_{k}\left( x\right) -\sin \left( \frac{3\left( k-n\right) \pi }{4%
}\right) \mathrm{kei}_{k}\left( x\right) \right] ,  \notag
\end{eqnarray}%
and%
\begin{eqnarray}
&&\left. \frac{\partial \,\mathrm{kei}_{\nu }\left( x\right) }{\partial \nu }%
\right\vert _{\nu =n}=-\frac{\pi }{2}\mathrm{ker}_{n}\left( x\right)
\label{DKein_Brychov} \\
&&+\frac{n!}{2}\sum_{k=0}^{n-1}\frac{\left( z/2\right) ^{k-n}}{k!\left(
n-k\right) }\left[ \sin \left( \frac{3\left( k-n\right) \pi }{4}\right) 
\mathrm{ker}_{k}\left( x\right) +\cos \left( \frac{3\left( k-n\right) \pi }{4%
}\right) \mathrm{kei}_{k}\left( x\right) \right] .  \notag
\end{eqnarray}

Nonetheless, integral representations and explicit closed-form expressions
for the order derivatives of the $\mathrm{ker}_{\nu }$ and $\,\mathrm{kei}%
_{\nu }$ functions for non-integral order are apparently absent in the
literature. Therefore, the main scope of this paper is to obtain alternative
closed-form expressions for the order derivatives of the $\mathrm{ber}_{\nu
} $ and $\,\mathrm{bei}_{\nu }$ functions, and new ones for the order
derivatives of the $\mathrm{ker}_{\nu }$ and $\,\mathrm{kei}_{\nu }$
functions for arbitrary real order, $\nu \in 
%TCIMACRO{\U{211d} }%
%BeginExpansion
\mathbb{R}
%EndExpansion
$. From these results, and taking into account the integral representations
given in (\ref{Dber_Int})\ and (\ref{Dbei_Int}), we will obtain the
calculation of some integrals involving the Kelvin functions $\mathrm{ber}%
_{\nu }$ and $\,\mathrm{bei}_{\nu }$.

This paper is organized as follows. Section \ref{Section: Order derivatives}%
\ is devoted to the calculation of the derivatives of the Kelvin functions
with respect to the order. First, we calculate the corresponding order
derivatives for non-negative order, and then, using reflection formulas, we
calculate the case of negative order. In Section \ref{Section: Integrals},
we calculate some integrals involving the $\mathrm{ber}_{\nu }$ and $\,%
\mathrm{bei}_{\nu }$ functions, which do not seem to be reported in the
literature. Finally, we collect the conclusions in Section \ref{Section:
Conclusions}. Also, we provide an alternative derivation for the integral
representation of the $\mathrm{ber}$ and $\,\mathrm{bei}$ functions in the
Appendix.

\section{Order derivatives of Kelvin functions\label{Section: Order
derivatives}}

According to the definitions given in (\ref{ber_nu_def})\ and (\ref%
{bei_nu_def}), $\forall x,\nu \geq 0$ we have that 
\begin{eqnarray}
\mathrm{ber}_{\nu }\left( x\right) &=&\mathrm{Re}\left[ e^{i\pi \nu }J_{\nu
}\left( e^{-i\pi /4}x\right) \right] ,  \label{ber_nu_Re} \\
\mathrm{bei}_{\nu }\left( x\right) &=&\mathrm{Im}\left[ e^{i\pi \nu }J_{\nu
}\left( e^{-i\pi /4}x\right) \right] ,  \label{bei_nu_Im}
\end{eqnarray}%
and%
\begin{eqnarray}
\mathrm{ker}_{\nu }\left( x\right) &=&\mathrm{Re}\left[ e^{-i\pi \nu
/2}K_{\nu }\left( e^{i\pi /4}x\right) \right] ,  \label{ker_nu_Re} \\
\mathrm{kei}_{\nu }\left( x\right) &=&\mathrm{Im}\left[ e^{-i\pi \nu
/2}K_{\nu }\left( e^{i\pi /4}x\right) \right] .  \label{kei_nu_Im}
\end{eqnarray}

We can rewrite (\ref{ber_nu_Re})-(\ref{kei_nu_Im}), using the fact that 
\begin{eqnarray}
\mathrm{Re}\,z &=&\frac{z+\bar{z}}{2},  \label{Re_z_def} \\
\mathrm{Im}\,z &=&\frac{z-\bar{z}}{2i},  \label{Im_z_def}
\end{eqnarray}%
and the properties $\forall \nu \in 
%TCIMACRO{\U{211d} }%
%BeginExpansion
\mathbb{R}
%EndExpansion
$ \cite[Eqn. 10.11.8\&10.34.7]{DLMF}%
\begin{eqnarray*}
\overline{J_{\nu }\left( z\right) } &=&J_{\nu }\left( \bar{z}\right) , \\
\overline{K_{\nu }\left( z\right) } &=&K_{\nu }\left( \bar{z}\right) ,
\end{eqnarray*}%
thereby%
\begin{eqnarray*}
\mathrm{ber}_{\nu }\left( x\right) &=&\frac{1}{2}\left[ e^{i\pi \nu }J_{\nu
}\left( e^{-i\pi /4}x\right) +e^{-i\pi \nu }J_{\nu }\left( e^{i\pi
/4}x\right) \right] , \\
\mathrm{bei}_{\nu }\left( x\right) &=&\frac{i}{2}\left[ e^{-i\pi \nu }J_{\nu
}\left( e^{i\pi /4}x\right) -e^{i\pi \nu }J_{\nu }\left( e^{-i\pi
/4}x\right) \right] ,
\end{eqnarray*}%
and%
\begin{eqnarray*}
\mathrm{ker}_{\nu }\left( x\right) &=&\frac{1}{2}\left[ e^{i\pi \nu
/2}K_{\nu }\left( e^{i\pi /4}x\right) +e^{i\pi \nu /2}K_{\nu }\left(
e^{-i\pi /4}x\right) \right] , \\
\mathrm{kei}_{\nu }\left( x\right) &=&\frac{i}{2}\left[ e^{i\pi \nu
/2}K_{\nu }\left( e^{-i\pi /4}x\right) -e^{-i\pi \nu /2}K_{\nu }\left(
e^{i\pi /4}x\right) \right] .
\end{eqnarray*}

Thus, taking into account (\ref{ber_nu_Re})-(\ref{Im_z_def}), the derivative
of the Kelvin functions with respect to the order can be expressed as
follows:

\begin{theorem}
$\forall x,\nu \geq 0$, the order derivatives of the Kevin functions are 
\begin{eqnarray}
\frac{\partial \,\mathrm{ber}_{\nu }\left( x\right) }{\partial \nu } &=&%
\mathrm{Re}\left[ e^{i\pi \nu }\frac{\partial J_{\nu }}{\partial \nu }\left(
e^{-i\pi /4}x\right) \right] -\pi \,\mathrm{bei}_{\nu }\left( x\right) ,
\label{D_ber_nu_JL} \\
\frac{\partial \,\mathrm{bei}_{\nu }\left( x\right) }{\partial \nu } &=&%
\mathrm{Im}\left[ e^{i\pi \nu }\frac{\partial J_{\nu }}{\partial \nu }\left(
e^{-i\pi /4}x\right) \right] +\pi \,\mathrm{ber}_{\nu }\left( x\right) ,
\label{D_bei_nu_JL}
\end{eqnarray}%
and%
\begin{eqnarray}
\frac{\partial \,\mathrm{ker}_{\nu }\left( x\right) }{\partial \nu } &=&%
\mathrm{Re}\left[ e^{-i\pi \nu /2}\frac{\partial K_{\nu }}{\partial \nu }%
\left( e^{i\pi /4}x\right) \right] +\frac{\pi }{2}\,\mathrm{kei}_{\nu
}\left( x\right) ,  \label{D_ker_nu_JL} \\
\frac{\partial \,\mathrm{kei}_{\nu }\left( x\right) }{\partial \nu } &=&%
\mathrm{Im}\left[ e^{-i\pi \nu /2}\frac{\partial K_{\nu }}{\partial \nu }%
\left( e^{i\pi /4}x\right) \right] -\frac{\pi }{2}\,\mathrm{ker}_{\nu
}\left( x\right) .  \label{D_kei_nu_JL}
\end{eqnarray}
\end{theorem}

\bigskip

For negative order $\nu $, we can use the following reflection formulas \cite%
[Eqn. 10.61.6\&7]{DLMF}:%
\begin{eqnarray}
\mathrm{ber}_{-\nu }\left( x\right) &=&\cos \pi \nu \,\mathrm{ber}_{\nu
}\left( x\right) +\sin \pi \nu \,\mathrm{bei}_{\nu }\left( x\right) +\frac{2%
}{\pi }\sin \pi \nu \,\mathrm{ker}_{\nu }\left( x\right) ,
\label{ber_-nu_def} \\
\mathrm{bei}_{-\nu }\left( x\right) &=&-\sin \pi \nu \,\mathrm{ber}_{\nu
}\left( x\right) +\cos \pi \nu \,\mathrm{bei}_{\nu }\left( x\right) +\frac{2%
}{\pi }\sin \pi \nu \,\mathrm{kei}_{\nu }\left( x\right) ,
\label{bei_-nu_def}
\end{eqnarray}%
and%
\begin{eqnarray}
\mathrm{ker}_{-\nu }\left( x\right) &=&\cos \pi \nu \,\mathrm{ker}_{\nu
}\left( x\right) -\sin \pi \nu \,\mathrm{kei}_{\nu }\left( x\right) ,
\label{ker_-nu_def} \\
\mathrm{kei}_{-\nu }\left( x\right) &=&\sin \pi \nu \,\mathrm{ker}_{\nu
}\left( x\right) +\cos \pi \nu \,\mathrm{kei}_{\nu }\left( x\right) .
\label{kei_-nu_def}
\end{eqnarray}

\bigskip

\begin{theorem}
$\forall x,\nu >0$, the following order derivatives hold true:%
\begin{eqnarray}
&&\frac{\partial \,\mathrm{ber}_{-\nu }\left( x\right) }{\partial \nu }
\label{Dber_-nu_resultado} \\
&=&-\mathrm{Re}\left[ e^{i\pi \nu /2}\left\{ \left( e^{-i\pi \nu }+\cos \pi
\nu \right) K_{\nu }\left( e^{i\pi /4}x\right) +\frac{2}{\pi }\sin \pi \nu 
\frac{\partial K_{\nu }}{\partial \nu }\left( e^{i\pi /4}x\right) \right\}
\right.  \notag \\
&&+\left. \frac{\partial J_{\nu }}{\partial \nu }\left( e^{-i\pi /4}x\right) %
\right] ,  \notag
\end{eqnarray}%
and%
\begin{eqnarray}
&&\frac{\partial \,\mathrm{ber}_{-\nu }\left( x\right) }{\partial \nu }
\label{Dbei_-nu_resultado} \\
&=&-\mathrm{Im}\left[ e^{i\pi \nu /2}\left\{ \left( e^{-i\pi \nu }+\cos \pi
\nu \right) K_{\nu }\left( e^{i\pi /4}x\right) +\frac{2}{\pi }\sin \pi \nu 
\frac{\partial K_{\nu }}{\partial \nu }\left( e^{i\pi /4}x\right) \right\}
\right.  \notag \\
&&+\left. \frac{\partial J_{\nu }}{\partial \nu }\left( e^{-i\pi /4}x\right) %
\right] .  \notag
\end{eqnarray}
\end{theorem}

\begin{proof}
Performing the order derivative in (\ref{ber_-nu_def}) and taking into
account (\ref{D_ber_nu_JL})\ and (\ref{D_bei_nu_JL}), we arrive at%
\begin{eqnarray}
&&-\frac{\partial \,\mathrm{ber}_{-\nu }\left( x\right) }{\partial \nu }
\label{Dber-nu_1} \\
&=&\sin \pi \nu \,\mathrm{Im}\left[ e^{i\pi \nu }\frac{\partial J_{\nu }}{%
\partial \nu }\left( e^{-i\pi /4}x\right) \right] +\cos \pi \nu \,\mathrm{Re}%
\left[ e^{i\pi \nu }\frac{\partial J_{\nu }}{\partial \nu }\left( e^{-i\pi
/4}x\right) \right]  \notag \\
&&+2\cos \pi \nu \,\mathrm{ker}_{\nu }\left( x\right) +\frac{2}{\pi }\sin
\pi \nu \frac{\partial \,\mathrm{ker}_{\nu }\left( x\right) }{\partial \nu }.
\notag
\end{eqnarray}%
Consider now (\ref{D_ker_nu_JL}) and the property 
\begin{equation}
\cos \pi \nu \,\mathrm{Re\,}z\pm \sin \pi \nu \,\mathrm{Im\,}z=\mathrm{Re}%
\left( e^{\mp i\pi \nu }z\right) ,  \label{property_Re}
\end{equation}%
to rewrite (\ref{Dber-nu_1})\ as%
\begin{eqnarray*}
-\frac{\partial \,\mathrm{ber}_{-\nu }\left( x\right) }{\partial \nu } &=&%
\mathrm{Re}\left[ \frac{\partial J_{\nu }}{\partial \nu }\left( e^{-i\pi
/4}x\right) \right] +2\cos \pi \nu \,\mathrm{ker}_{\nu }\left( x\right) \\
&&+\sin \pi \nu \,\mathrm{kei}_{\nu }\left( x\right) +\frac{2}{\pi }\sin \pi
\nu \,\mathrm{Re}\left[ e^{-i\pi \nu /2}\frac{\partial K_{\nu }}{\partial
\nu }\left( e^{i\pi /4}x\right) \right] .
\end{eqnarray*}%
Finally, substitute (\ref{ker_nu_Re})\ and (\ref{kei_nu_Im}), and apply
again (\ref{property_Re})\ to arrive at (\ref{Dber_-nu_resultado}). We can
perform a similar proof for (\ref{Dbei_-nu_resultado}), applying the
property 
\begin{equation}
\cos \pi \nu \,\mathrm{Im\,}z\pm \sin \pi \nu \,\mathrm{Re\,}z=\mathrm{Im}%
\left( e^{\pm i\pi \nu }z\right) .  \label{Property_Im}
\end{equation}
\end{proof}

\begin{theorem}
$\forall x,\nu >0$, the following order derivatives hold true:%
\begin{equation}
\frac{\partial \,\mathrm{ker}_{-\nu }\left( x\right) }{\partial \nu }=\frac{%
\pi }{2}\mathrm{Im}\left[ e^{i\pi \nu /2}K_{\nu }\left( e^{i\pi /4}x\right) %
\right] -\mathrm{Re}\left[ e^{i\pi \nu /2}\frac{\partial K_{\nu }}{\partial
\nu }\left( e^{i\pi /4}x\right) \right] ,  \label{Dker_-nu_resultado}
\end{equation}%
and%
\begin{equation}
\frac{\partial \,\mathrm{kei}_{-\nu }\left( x\right) }{\partial \nu }=-\frac{%
\pi }{2}\mathrm{Re}\left[ e^{i\pi \nu /2}K_{\nu }\left( e^{i\pi /4}x\right) %
\right] -\mathrm{Im}\left[ e^{i\pi \nu /2}\frac{\partial K_{\nu }}{\partial
\nu }\left( e^{i\pi /4}x\right) \right] .  \label{Dkei_-nu_resultado}
\end{equation}
\end{theorem}

\begin{proof}
Perform in (\ref{ker_-nu_def}) the derivative with respect to the order,
taking into account (\ref{D_ker_nu_JL})\ and (\ref{D_kei_nu_JL}), to arrive
at%
\begin{eqnarray*}
-\frac{\partial \,\mathrm{ker}_{-\nu }\left( x\right) }{\partial \nu } &=&-%
\frac{\pi }{2}\sin \pi \nu \,\mathrm{ker}_{\nu }\left( x\right) +\cos \pi
\nu \,\mathrm{Re}\left[ e^{-i\pi \nu /2}\frac{\partial K_{\nu }}{\partial
\nu }\left( e^{i\pi /4}x\right) \right] \\
&&-\frac{\pi }{2}\cos \pi \nu \,\mathrm{kei}_{\nu }\left( x\right) -\sin \pi
\nu \,\mathrm{Im}\left[ e^{-i\pi \nu /2}\frac{\partial K_{\nu }}{\partial
\nu }\left( e^{i\pi /4}x\right) \right] .
\end{eqnarray*}%
Substitute now (\ref{ker_nu_Re})\ and (\ref{kei_nu_Im})\ and apply the
properties (\ref{ker_nu_Re}) and (\ref{Property_Im}) to obtain (\ref%
{Dker_-nu_resultado}). Similarly, we can derive (\ref{Dkei_-nu_resultado}).
\end{proof}

\bigskip

Therefore, according to the results of the above theorems, knowing $\partial
J_{\nu }/\partial \nu $ and $\partial K_{\nu }/\partial \nu $ in
closed-form, we can express the order derivatives of the Kelvin functions in
closed-form as well. Recently in \cite{DBesselJL}, the order derivative of
the Bessel function of the first kind $\forall \nu >0$, $z\in 
%TCIMACRO{\U{2102} }%
%BeginExpansion
\mathbb{C}
%EndExpansion
$, $z\neq 0$, is expressed in closed-form as follows:\ 
\begin{eqnarray}
&&\frac{\partial J_{\nu }\left( z\right) }{\partial \nu }
\label{DJnu_closed_form} \\
&=&\frac{-\pi J_{-\nu }\left( z\right) \csc \pi \nu }{2\Gamma ^{2}\left( \nu
+1\right) }\left( \frac{z}{2}\right) ^{2\nu }\,_{2}F_{3}\left( \left. 
\begin{array}{c}
\nu ,\nu +\frac{1}{2} \\ 
\nu +1,\nu +1,2\nu +1%
\end{array}%
\right\vert -z^{2}\right)  \notag \\
&&-J_{\nu }\left( z\right) \left[ \frac{z^{2}}{4\left( 1-\nu ^{2}\right) }%
\,_{3}F_{4}\left( \left. 
\begin{array}{c}
1,1,\frac{3}{2} \\ 
2,2,2-\nu ,2+\nu%
\end{array}%
\right\vert -z^{2}\right) \right.  \notag \\
&&\quad \left. 
\begin{array}{c}
%TCIMACRO{\TeXButton{TeX field}{\displaystyle} }%
%BeginExpansion
\displaystyle
%EndExpansion
\\ 
%TCIMACRO{\TeXButton{TeX field}{\displaystyle}}%
%BeginExpansion
\displaystyle%
%EndExpansion
\end{array}%
+\log \left( \frac{2}{z}\right) +\frac{1}{2\nu }+\psi \left( \nu \right) %
\right] .  \notag
\end{eqnarray}

However, (\ref{DJnu_closed_form})\ cannot be used for non-negative integral
orders, i.e. $\nu =0,1,2,\ldots $. A way to avoid this inconvenient is to
express the derivative of $J_{\nu }\left( z\right) $ with respect to the
order using the Meijer-$G$ function \cite{DBesselJL}, 
\begin{eqnarray}
\frac{\partial J_{\nu }\left( z\right) }{\partial \nu } &=&\frac{\pi }{2}%
\left[ Y_{\nu }\left( z\right) \frac{\left( z/2\right) ^{2\nu }}{\,\Gamma
^{2}\left( \nu +1\right) }\,_{2}F_{3}\left( \left. 
\begin{array}{c}
\nu ,1/2+\nu \\ 
2\nu +1,\nu +1,\nu +1%
\end{array}%
\right\vert -z^{2}\right) \right.  \label{DJnu_Meijer} \\
&&-\left. \frac{\nu J_{\nu }\left( z\right) }{\sqrt{\pi }}%
\,G_{2,4}^{3,0}\left( z^{2}\left\vert 
\begin{array}{c}
1/2,1 \\ 
0,0,\nu ,-\nu%
\end{array}%
\right. \right) \right] ,\quad \nu \geq 0,\,\mathrm{Re\,}z>0.  \notag
\end{eqnarray}

Similarly, for the order derivative of the Macdonald function, $\forall \nu
>0$, $z\in 
%TCIMACRO{\U{2102} }%
%BeginExpansion
\mathbb{C}
%EndExpansion
$, $z\neq 0$, we have \cite{DBesselJL}, 
\begin{eqnarray}
&&\frac{\partial K_{\nu }\left( z\right) }{\partial \nu }
\label{DKnu_resultado} \\
&=&\frac{\pi }{2}\csc \pi \nu \left\{ \pi \cot \pi \nu \,I_{\nu }\left(
z\right) -\left[ I_{\nu }\left( z\right) +I_{-\nu }\left( z\right) \right] 
\begin{array}{c}
%TCIMACRO{\TeXButton{TeX field}{\displaystyle} }%
%BeginExpansion
\displaystyle
%EndExpansion
\\ 
%TCIMACRO{\TeXButton{TeX field}{\displaystyle}}%
%BeginExpansion
\displaystyle%
%EndExpansion
\end{array}%
\right.  \notag \\
&&\left. \left[ \frac{z^{2}}{4\left( 1-\nu ^{2}\right) }\,_{3}F_{4}\left(
\left. 
\begin{array}{c}
1,1,\frac{3}{2} \\ 
2,2,2-\nu ,2+\nu%
\end{array}%
\right\vert z^{2}\right) +\log \left( \frac{z}{2}\right) -\psi \left( \nu
\right) -\frac{1}{2\nu }\right] \right\}  \notag \\
&&+\frac{1}{4}\left\{ I_{-\nu }\left( z\right) \Gamma ^{2}\left( \nu \right)
\left( \frac{z}{2}\right) ^{2\nu }\,_{2}F_{3}\left( \left. 
\begin{array}{c}
\nu ,\frac{1}{2}+\nu \\ 
1+\nu ,1+\nu ,1+2\nu%
\end{array}%
\right\vert z^{2}\right) \right.  \notag \\
&&\quad -\left. I_{\nu }\left( z\right) \Gamma ^{2}\left( -\nu \right)
\left( \frac{z}{2}\right) ^{-2\nu }\,_{2}F_{3}\left( \left. 
\begin{array}{c}
\nu ,\frac{1}{2}-\nu \\ 
1-\nu ,1-\nu ,1-2\nu%
\end{array}%
\right\vert z^{2}\right) \right\} .  \notag
\end{eqnarray}

Again, (\ref{DKnu_resultado})\ cannot be used for non-negative integral
orders, as well as for non-negative half-integral orders, i.e. $\nu
=0,1/2,1,3/2,\ldots $. In these cases, we can employ the following
representation, which uses Meijer-$G$ functions: 
\begin{eqnarray}
&&\frac{\partial K_{\nu }\left( z\right) }{\partial \nu }
\label{DKnu_Meijer} \\
&=&\frac{\nu }{2}\left[ \frac{K_{\nu }\left( z\right) }{\sqrt{\pi }}%
\,G_{2,4}^{3,1}\left( z^{2}\left\vert 
\begin{array}{c}
1/2,1 \\ 
0,0,\nu ,-\nu%
\end{array}%
\right. \right) -\sqrt{\pi }I_{\nu }\left( z\right) G_{2,4}^{4,0}\left(
z^{2}\left\vert 
\begin{array}{c}
1/2,1 \\ 
0,0,\nu ,-\nu%
\end{array}%
\right. \right) \right] ,  \notag \\
\quad &&\nu \geq 0,\,\mathrm{Re\,}z>0.  \notag
\end{eqnarray}

It is worth noting that the numerical evaluation of (\ref{DJnu_closed_form})
and (\ref{DKnu_resultado})\ is $\approx 10$ times faster than (\ref%
{DJnu_Meijer}) and (\ref{DKnu_Meijer}),\ respectively.

\section{Application to the calculation of some integrals\label{Section:
Integrals}}

Comparing the integral representations given in the Introduction to the
results given in Section \ref{Section: Order derivatives}, we can calculate
some integrals involving Kelvin functions that do not seem to be reported in
the literature.

\begin{theorem}
The following integrals hold true:%
\begin{eqnarray}
&&\int_{0}^{1}u^{\nu +1}\log \left( 1-u^{2}\right) f_{\nu }\left( xu\right)
du  \label{Int_ber_bei_resultado} \\
&=&\frac{1}{\sqrt{2}x}\left\{ \left[ \frac{\pi }{4}+\log \left( \frac{x}{2}%
\right) +\gamma \right] f_{\nu +1}\left( x\right) \right.  \notag \\
&&\pm \left. \left[ \frac{\pi }{4}-\log \left( \frac{x}{2}\right) -\gamma %
\right] g_{\nu +1}\left( x\right) +\sqrt{2}\,\mathrm{Re}\left[ e^{i\pi
\left( \nu \pm 1/4\right) }\frac{\partial J_{\nu }}{\partial \nu }\left(
e^{-i\pi /4}x\right) \right] \right\} ,  \notag
\end{eqnarray}%
where $f_{\nu }$ and $g_{\nu }$ denote the ordered pair $\mathrm{ber}_{\nu }$
and $\mathrm{bei}_{\nu }$.
\end{theorem}

\begin{proof}
If we equate (\ref{Dber_Int})\ to (\ref{D_ber_nu_JL}), and then we shift the
order $\nu \rightarrow \nu +1$, and we change the integration variable $%
u\rightarrow u^{2}$, we obtain%
\begin{eqnarray}
&&\int_{0}^{1}u^{\nu /2}\left[ \gamma +\log \left( 1-u^{2}\right) \right] %
\left[ \mathrm{ber}_{\nu }\left( x\,u\right) +\mathrm{bei}_{\nu }\left(
x\,u\right) \right] du  \label{Int_ber} \\
&=&\frac{\sqrt{2}}{x}\left\{ \frac{\pi }{4}\,\mathrm{bei}_{\nu +1}\left(
x\right) +\log \left( \frac{x}{2}\right) \mathrm{ber}_{\nu +1}\left(
x\right) -\mathrm{Re}\left[ e^{i\pi \left( \nu +1\right) }\frac{\partial
J_{\nu }}{\partial \nu }\left( e^{-i\pi /4}x\right) \right] \right\} . 
\notag
\end{eqnarray}%
Similarly, from (\ref{Dbei_Int})\ and (\ref{D_bei_nu_JL}), we arrive at%
\begin{eqnarray}
&&\int_{0}^{1}u^{\nu /2}\left[ \gamma +\log \left( 1-u^{2}\right) \right] %
\left[ \mathrm{ber}_{\nu }\left( x\,u\right) +\mathrm{bei}_{\nu }\left(
x\,u\right) \right] du  \label{Int_bei} \\
&=&\frac{\sqrt{2}}{x}\left\{ \frac{\pi }{4}\,\mathrm{ber}_{\nu +1}\left(
x\right) -\log \left( \frac{x}{2}\right) \mathrm{bei}_{\nu +1}\left(
x\right) +\mathrm{Im}\left[ e^{i\pi \left( \nu +1\right) }\frac{\partial
J_{\nu }}{\partial \nu }\left( e^{-i\pi /4}x\right) \right] \right\} . 
\notag
\end{eqnarray}%
Summing up (\ref{Int_ber})\ and (\ref{Int_bei}), and taking into account the
property 
\begin{equation*}
\mathrm{Im}\,z-\mathrm{Re}\,z=\sqrt{2}\mathrm{Re}\,\left( e^{-i3\pi
/4}z\right) ,
\end{equation*}%
we have%
\begin{eqnarray}
&&\int_{0}^{1}u^{\nu /2}\left[ \gamma +\log \left( 1-u^{2}\right) \right] 
\mathrm{ber}_{\nu }\left( x\,u\right) du  \label{Int_ber_2} \\
&=&\frac{1}{\sqrt{2}x}\left\{ \left[ \frac{\pi }{4}+\log \left( \frac{x}{2}%
\right) \right] \mathrm{ber}_{\nu +1}\left( x\right) +\left[ \frac{\pi }{4}%
-\log \left( \frac{x}{2}\right) \right] \mathrm{bei}_{\nu +1}\left( x\right)
\right.  \notag \\
&&\qquad +\left. \sqrt{2}\mathrm{Re}\left[ e^{i\pi \left( \nu +1/4\right) }%
\frac{\partial J_{\nu }}{\partial \nu }\left( e^{-i\pi /4}x\right) \right]
\right\} .  \notag
\end{eqnarray}%
Similarly, subtracting (\ref{Int_ber})\ and (\ref{Int_bei}), and taking into
account the property%
\begin{equation*}
\mathrm{Im}\,z+\mathrm{Re}\,z=\sqrt{2}\mathrm{Re}\,\left( e^{-i\pi
/4}z\right) ,
\end{equation*}%
we have%
\begin{eqnarray}
&&\int_{0}^{1}u^{\nu /2}\left[ \gamma +\log \left( 1-u^{2}\right) \right] 
\mathrm{bei}_{\nu }\left( x\,u\right) du  \label{Int_bei_2} \\
&=&\frac{1}{\sqrt{2}x}\left\{ \left[ -\frac{\pi }{4}+\log \left( \frac{x}{2}%
\right) \right] \mathrm{ber}_{\nu +1}\left( x\right) +\left[ \frac{\pi }{4}%
+\log \left( \frac{x}{2}\right) \right] \mathrm{bei}_{\nu +1}\left( x\right)
\right.  \notag \\
&&\qquad +\left. \sqrt{2}\mathrm{Re}\left[ e^{i\pi \left( \nu -1/4\right) }%
\frac{\partial J_{\nu }}{\partial \nu }\left( e^{-i\pi /4}x\right) \right]
\right\} .  \notag
\end{eqnarray}%
Considering now the following indefinite integrals \cite[Eqn. 10.71.1]{DLMF},%
\begin{eqnarray}
\int u^{\nu +1}\mathrm{ber}_{\nu }\left( u\right) du &=&\frac{u^{\nu +1}}{%
\sqrt{2}}\left[ \mathrm{bei}_{\nu +1}\left( x\right) -\mathrm{ber}_{\nu
+1}\left( x\right) \right] ,  \label{Int_ber_NIST} \\
\int u^{\nu +1}\mathrm{bei}_{\nu }\left( u\right) du &=&-\frac{u^{\nu +1}}{%
\sqrt{2}}\left[ \mathrm{bei}_{\nu +1}\left( x\right) +\mathrm{ber}_{\nu
+1}\left( x\right) \right] ,  \label{Int_bei_NIST}
\end{eqnarray}%
we obtain (\ref{Int_ber_bei_resultado}) with $f_{\nu }=\mathrm{ber}_{\nu }$
and $g_{\nu }=\mathrm{bei}_{\nu }$ from (\ref{Int_ber_2})\ and (\ref%
{Int_ber_NIST});\ and (\ref{Int_ber_bei_resultado}) with $f_{\nu }=\mathrm{%
bei}_{\nu }$ and $g_{\nu }=\mathrm{ber}_{\nu }$ from (\ref{Int_bei_2}) and (%
\ref{Int_bei_NIST}), as we wanted to prove.
\end{proof}

\section{Conclusions\label{Section: Conclusions}}

We have obtained closed-form expressions of the derivatives of the Kelvin
functions with respect to the order in (\ref{D_ber_nu_JL})-(\ref{D_kei_nu_JL}%
) for $\nu \in 
%TCIMACRO{\U{211d} }%
%BeginExpansion
\mathbb{R}
%EndExpansion
$. These expressions are based on recent closed-form expressions found in
the literature for $\partial J_{\nu }/\partial \nu $ and $\partial K_{\nu
}/\partial \nu $, i.e. (\ref{DJnu_closed_form})-(\ref{DKnu_Meijer}).

On the one hand, unlike the order derivatives of the $\mathrm{ber}_{\nu }$
and $\,\mathrm{bei}_{\nu }$ functions found in the literature, i.e. (\ref%
{Dber_Brychov})\ and (\ref{Dbei_Brychov}), the results given in (\ref%
{D_ber_nu_JL})\ and (\ref{D_bei_nu_JL}) are much more simple. Moreover, the
expressions obtained for negative orders given in (\ref{Dber_-nu_resultado})
and (\ref{Dbei_-nu_resultado}), in combination with (\ref{DJnu_Meijer})\ and
(\ref{DKnu_Meijer}), allow us the evaluation of the order derivatives of the 
$\mathrm{ber}_{\nu }$ and $\,\mathrm{bei}_{\nu }$ functions for negative
integral orders. The latter cannot be performed with the formulas found in
the literature, i.e. (\ref{DBern_Brychov})\ and (\ref{DBein_Brychov}). Also,
the derivatives of the $\mathrm{ker}_{\nu }$, and $\,\mathrm{kei}_{\nu }$
functions with respect to the order given in (\ref{D_ker_nu_JL})\ and (\ref%
{D_kei_nu_JL}) for non-negative orders, and the ones given in (\ref%
{Dker_-nu_resultado}) and (\ref{Dkei_-nu_resultado})\ for negative orders,
seem to be novel.

On the other hand, we have calculated the integrals given in (\ref%
{Int_ber_bei_resultado}). These integrals do not seem to be reported in
closed-form in the literature. Finally, in the Appendix, we provide an
alternative derivation of the integral representations of the $\mathrm{ber}$
and $\,\mathrm{bei}$ functions found in the literature.

\paragraph{Acknowledgements}

It is a pleasure to thank Prof. A. Apelblat for the literature and wise
comments given to the author.

\appendix

\section{Integral representations}

\begin{theorem}
The following integral representations hold true:%
\begin{eqnarray}
\mathrm{ber}\left( x\right) &=&\frac{2}{\pi }\int_{0}^{\pi /2}\cosh \left( 
\frac{x\,\mathrm{sc\,}\theta }{\sqrt{2}}\right) \cos \left( \frac{x\,\mathrm{%
sc\,}\theta }{\sqrt{2}}\right) du,  \label{ber(x)_Int} \\
\mathrm{bei}\left( x\right) &=&\frac{2}{\pi }\int_{0}^{\pi /2}\sinh \left( 
\frac{x\,\mathrm{sc\,}\theta }{\sqrt{2}}\right) \sin \left( \frac{x\,\mathrm{%
sc\,}\theta }{\sqrt{2}}\right) du,  \label{bei(x)_Int}
\end{eqnarray}%
where $\mathrm{sc\,}\theta =\sin \theta $ or $\cos \theta $.
\end{theorem}

\begin{proof}
Multiply by $1/s$ the trigonometric identity,%
\begin{equation}
\cos A+\cos B=2\cos \left( \frac{A+B}{2}\right) \cos \left( \frac{A-B}{2}%
\right) ,  \label{cosA+cosB}
\end{equation}%
and take $A=a/s$, $B=b/s$ to arrive at%
\begin{equation}
\frac{1}{s}\cos \left( \frac{a}{s}\right) +\frac{1}{s}\cos \left( \frac{b}{s}%
\right) =2\frac{1}{\sqrt{s}}\cos \left( \frac{a+b}{2s}\right) \frac{1}{\sqrt{%
s}}\cos \left( \frac{a-b}{2\sqrt{s}}\right) .  \label{Int_proof_1}
\end{equation}%
Defining the following functions $f$, $g$, $f_{1}$ and $g_{1}$ via the
Laplace transform as follows, 
\begin{eqnarray*}
\mathcal{L}\left[ f\right] &=&\frac{1}{s}\cos \left( \frac{a}{s}\right) , \\
\mathcal{L}\left[ g\right] &=&\frac{1}{s}\cos \left( \frac{b}{s}\right) , \\
\mathcal{L}\left[ f_{1}\right] &=&\frac{1}{s}\cos \left( \frac{a+b}{s}%
\right) , \\
\mathcal{L}\left[ g_{1}\right] &=&\frac{1}{s}\cos \left( \frac{a-b}{s}%
\right) ,
\end{eqnarray*}%
and applying the Laplace anti-transform to (\ref{Int_proof_1}), we obtain%
\begin{eqnarray}
f+g &=&2\int_{0}^{t}f_{1}\left( t-\tau \right) g_{1}\left( \tau \right) d\tau
\label{f+g=Convolution} \\
&=&2\int_{0}^{t}f_{1}\left( \tau \right) g_{1}\left( t-\tau \right) d\tau ,
\label{f+g=Convolution2}
\end{eqnarray}%
where we have applied the convolution theorem of the Laplace transform \cite[%
Eqn. 17.12.5]{Gradshteyn}. Since, according to \cite[Eqn. 2.4.1(2)\&(3)]%
{Prudnikov5} we have%
\begin{eqnarray*}
\mathcal{L}^{-1}\left[ \frac{1}{s}\cos \left( \frac{\alpha }{s}\right) %
\right] &=&\mathrm{ber}\left( 2\sqrt{\alpha t}\right) , \\
\mathcal{L}^{-1}\left[ \frac{1}{\sqrt{s}}\cos \left( \frac{\alpha }{s}%
\right) \right] &=&\frac{1}{\sqrt{\pi t}}\cosh \sqrt{2\alpha t}\cos \sqrt{%
2\alpha t},
\end{eqnarray*}%
we rewrite (\ref{f+g=Convolution})\ as%
\begin{eqnarray}
&&\mathrm{ber}\left( 2\sqrt{at}\right) +\mathrm{ber}\left( 2\sqrt{bt}\right)
\label{ber(a)+ber(b)} \\
&=&\frac{2}{\pi }\int_{0}^{t}\frac{\cosh \sqrt{\left( a+b\right) \left(
t-\tau \right) }\cos \sqrt{\left( a+b\right) \left( t-\tau \right) }}{\sqrt{%
\tau \left( t-\tau \right) }}  \notag \\
&&\qquad \cosh \sqrt{\left( a-b\right) \tau }\cos \sqrt{\left( a-b\right)
\tau }\,d\tau .  \notag
\end{eqnarray}%
Taking $a=b$ and performing the changes of variables $x=2\sqrt{at}$ and $%
\tau =z^{2}/\left( 4a\right) $, (\ref{ber(a)+ber(b)})\ becomes 
\begin{equation*}
\mathrm{ber}\left( x\right) =\frac{1}{\pi }\int_{0}^{x}\frac{\cosh \sqrt{%
\frac{x^{2}-z^{2}}{2}}\cos \sqrt{\frac{x^{2}-z^{2}}{2}}}{\sqrt{x^{2}-z^{2}}}%
dz.
\end{equation*}%
Finally, performing the change $z=x\sin \theta $, we arrive at (\ref%
{ber(x)_Int})\ with $\mathrm{sc\,}\theta =\cos \theta $. To arrive at (\ref%
{ber(x)_Int})\ with $\mathrm{sc\,}\theta =\sin \theta $ we have to take (\ref%
{f+g=Convolution2})\ instead of (\ref{f+g=Convolution}) in the above
reasoning. To obtain (\ref{bei(x)_Int}), we can depart from the
trigonometric identity 
\begin{equation*}
\sin A+\sin B=2\sin \left( \frac{A+B}{2}\right) \sin \left( \frac{A-B}{2}%
\right) ,
\end{equation*}%
and consider the following Laplace anti-transforms \cite[Eqn. 2.4.1(2)\&(3)]%
{Prudnikov5} 
\begin{eqnarray*}
\mathcal{L}^{-1}\left[ \frac{1}{s}\sin \left( \frac{\alpha }{s}\right) %
\right] &=&\mathrm{bei}\left( 2\sqrt{\alpha t}\right) , \\
\mathcal{L}^{-1}\left[ \frac{1}{\sqrt{s}}\sin \left( \frac{\alpha }{s}%
\right) \right] &=&\frac{1}{\sqrt{\pi t}}\sinh \sqrt{2\alpha t}\sin \sqrt{%
2\alpha t}.
\end{eqnarray*}
\end{proof}

\begin{remark}
It is worth noting that we can derive the integral representations (\ref%
{ber(x)_Int}) and (\ref{bei(x)_Int})\ from the integral representations
given in (\ref{ber_nu_Int}) and (\ref{bei_nu_Int}). For instance, taking $%
\nu =0$ and changing the argument $x\sqrt{2}\rightarrow x$ of the Kelvin
function in (\ref{ber(x)_Int}),\ we have%
\begin{equation}
\mathrm{ber}\left( x\right) =\frac{1}{\pi }\int_{0}^{\pi }\cos \left( \frac{%
x\sin t}{\sqrt{2}}\right) \cosh \left( \frac{x\sin t}{\sqrt{2}}\right) dt.
\label{ber_Sin}
\end{equation}%
Performing the change of variables $u=t-\pi /2$ and taking into account the
parity of the integrand, we arrive at%
\begin{equation}
\mathrm{ber}\left( x\right) =\frac{2}{\pi }\int_{0}^{\pi /2}\cos \left( 
\frac{x\cos t}{\sqrt{2}}\right) \cosh \left( \frac{x\cos t}{\sqrt{2}}\right)
dt,  \label{ber_Cos}
\end{equation}%
which is equivalent to (\ref{ber(x)_Int})\ with $\mathrm{sc\,}\theta =\cos
\theta $. Now, split (\ref{ber_Sin})\ as 
\begin{eqnarray}
\mathrm{ber}\left( x\right) &=&\frac{1}{\pi }\int_{0}^{\pi /2}\cos \left( 
\frac{x\sin t}{\sqrt{2}}\right) \cosh \left( \frac{x\sin t}{\sqrt{2}}\right)
dt  \notag \\
&&+\frac{1}{\pi }\int_{\pi /2}^{\pi }\cos \left( \frac{x\sin t}{\sqrt{2}}%
\right) \cosh \left( \frac{x\sin t}{\sqrt{2}}\right) dt,  \label{Int_2}
\end{eqnarray}%
and perform in (\ref{Int_2})\ the change of variables $u=t-\pi /2$, taking
into account (\ref{ber_Cos}), to arrive at%
\begin{equation*}
\mathrm{ber}\left( x\right) =\frac{1}{\pi }\int_{0}^{\pi /2}\cos \left( 
\frac{x\sin t}{\sqrt{2}}\right) \cosh \left( \frac{x\sin t}{\sqrt{2}}\right)
dt+\frac{1}{2}\mathrm{ber}\left( x\right) ,
\end{equation*}%
which is equivalent to (\ref{ber(x)_Int})\ with $\mathrm{sc\,}\theta =\sin
\theta $. For the integral representation of the $\mathrm{bei}$ function, we
can perform a similar derivation.
\end{remark}

\end{document}